\documentclass[12pt,draft]{amsart}
\usepackage[usenames,dvipsnames]{color}
\usepackage{amssymb,fixme}
\usepackage{amscd,amsmath,amssymb,amsfonts}
\usepackage{amsmath} 
\usepackage{enumerate}
\usepackage[all,cmtip]{xy}
\usepackage[utf8]{inputenc}
\newcommand{\Gal}{\mathrm{Gal}}
\newcommand{\Br}{\mathrm{Br}}
\newcommand{\CH}{\mathrm{CH}}
\newcommand{\CHL}{\mathrm{CH}_\mathrm{L}}
\newcommand{\T}{\mathrm{T}}
\newcommand{\Q}{\mathbb{Q}}
\newcommand{\Z}{\mathbb{Z}}

\newcommand{\G}{\mathbb{G}}
\newcommand{\F}{\mathbb{F}}
\newcommand{\et}{\textrm{et}}

\newcommand{\tensor}{\otimes}

\renewcommand{\phi}{\varphi}
\renewcommand{\theta}{\vartheta}
\renewcommand{\epsilon}{\varepsilon}
\newcommand{\isom}{\cong}

\newcommand{\h}[1]{\mathrm{H}_{\operatorname{et}}^{#1}}

\newcommand{\hl}[1]{\mathrm{H}_\textrm{L}^{#1}}
\newcommand{\hm}[1]{\mathrm{H}_\textrm{M}^{#1}}

\newcommand{\coker}{\operatorname{coker}}

\newcommand{\GalGrp}{G}
\newcommand{\ch}{\operatorname{char}}

\newcommand{\Hom}{\operatorname{Hom}}

\newcommand{\tensorpow}[2]{#1^{\tensor#2}}
\newcommand{\TC}{\mathrm{TC}}

\newtheorem{theorem}{Theorem}[section]
\newtheorem*{theorem*}{Theorem}

\newtheorem{lemma}[theorem]{Lemma}
\newtheorem{proposition}[theorem]{Proposition}

\theoremstyle{remark}
\newtheorem{remark}[theorem]{Remark}


\title[Tate Conjecture and Higher Brauer Groups]{Tate Conjecture and Higher Brauer Groups of Abelian Varieties in Characteristic Zero}
\author{Thomas Jahn}
\address{Universit\"at Bielefeld, Fakult\"at f\"ur Mathematik, Postfach 100131, Universit\"atsstr.~25, 33501~Bielefeld, Germany}
\email{tjahn@math.uni-bielefeld.de}
\date{\today}

\begin{document}
\begin{abstract}
  Let \(A\) be an abelian variety over a field finitely generated over \(\Q\).
  We show that the finiteness of the \(\ell\)-primary torsion subgroup of the higher Brauer group is a sufficient criterion for the Tate conjecture to hold.
  Furthermore, we extend methods for computations of transcendental Brauer groups to higher Brauer groups.
\end{abstract}

\maketitle

\section{Introduction}
Let \(X\) be a smooth projective variety over a field \(k\).
The cohomological Brauer groups \(\Br(X)=\h2(X,\G_m)\) appears in several different settings in algebraic geometry.
For example it obstructs the Hasse principle on \(X\) (see \cite{Manin1970}) and it appears in a formula describing Zeta functions (see \cite[Conjecture C]{Tate1966}).

The list of applications also includes obstructions to the Tate conjecture:
Let \(k\) be finitely generated over its prime field and let \(X\) be a smooth projective \(k\)-variety.
The Tate conjecture \(\mathrm{TC}^r(X)_{\mathbb{Q}_\ell}\) in codimension \(r\) at a prime \(\ell\neq\ch(k)\) is the statement that the canonical cycle map \(c^r_{\Q_\ell}:\CH^r(X)\tensor\Q_\ell\to\h{2r}(\overline{X},\Q_\ell(r))^{\GalGrp_k}\) is surjective~\cite{Tate1994}.
Here \(\overline{X}=X\times_k\overline{k}\) denotes the base change to a separable closure \(\overline{k}\) of \(k\) and \(\GalGrp_k=\Gal(\overline{k}/k)\) is the absolute Galois group.
For a surface \(X\) over a finite field \(k\) Tate has shown that \(\mathrm{TC}^1(X)_{\mathbb{Q}_\ell}\) holds if and only if the \(\ell\)-primary torsion subgroup \(\Br(X)\{\ell\}\) of the cohomological Brauer group is finite \cite[Theorem~5.2]{Tate1966}.

This has been generalised using higher Brauer groups which are defined using Lichtenbaum cohomology.
The Lichtenbaum cohomology groups \(\hl{m}(X,\Z(n))\) are defined to be the hyper-cohomology groups of the (unbounded) complex of \'etale sheaves \(\Z_X(n)_\et\) given by Bloch's cycle complex \cite{Bloch1986} on the small \'etale site \(X_\et\).
The higher Brauer groups considered in this note are \(\Br^r(X)=\hl{2r+1}(X,\Z(r))\); as there exists a quasi-isomorphism \(\Z_X(1)_\et\sim\mathbb{G}_m[-1]\), we retrieve the Brauer group \(\Br(X)=\h2(X,\mathbb{G}_m)\isom\hl3(X,\Z(1))=\Br^1(X)\).
For any smooth projective, geometrically integral variety \(X\) over a finite field \(k\) and any prime \(\ell\neq\ch(k)\), the Tate conjecture \(\mathrm{TC}^r(X)_{\Q_\ell}\) holds precisely when the \(\ell\)-primary subgroup of \(\Br^r(X)\) is finite \cite[Theorem\,1.4]{RosenschonSrinivas2013}.

We give a partial answer to the question whether higher Brauer groups provide a similar obstruction to the Tate conjecture when \(k\) is finitely generated over \(\Q\) in Section\,\ref{sec:tateobst} where we prove the following
\begin{theorem*}
  Let \(A\) be an abelian variety over a field \(k\) finitely generated over \(\Q\).
  Fix an integer \(r\) and a prime \(\ell>2r\).
  If the group \(\Br^r(A)\{\ell\}\) is finite, then the Tate conjecture \(\mathrm{TC}^r(A)_{\Q_\ell}\) for \(A\) in codimension \(r\) holds at the prime \(\ell\).
\end{theorem*}

For any smooth projective variety \(X\) over a field \(k\) there exists a natural homomorphism \(\Br^r(X)\to\Br^r(\overline{X})\) where again \(\overline{X}=X\times_k\overline{k}\) denotes the base change to a separable closure \(\overline{k}\) of \(k\).
The kernel \(\Br_\textrm{alg}^r(X):=\ker\left(\Br^r(X)\to\Br^r\left(\overline{X}\right)\right)\) is referred to as \emph{algebraic Brauer group} while the quotient \(\Br^r(X)/\Br^r_\textrm{alg}(X)\) is called \emph{transcendental Brauer group}.

We remark that if the characteristic if \(k\) is zero and \(X\) is geometrically integral, the cokernel of \(\Br(X)\to\Br(\overline{X})\) is finite \cite[Th\'eor\`eme\,2.1]{CollioTSkorobogatov2013}.
The question whether this is also true for \(r\geq2\) is open.

If furthermore \(X=A\) is an abelian variety, Skorobogatov and Zarhin have proven (in the classical case, i.e. \(r=1\)) the formula 
\[\Br(A)[n]/\Br_\textrm{alg}(A)[n]\isom\h2(\overline{A},\mu_n)^{\GalGrp_k}/(\CH^1(\overline{A})\tensor\Z/n\Z)^{\GalGrp_k}\]
for any odd integer \(n\) \cite[Corollary\,2.3]{SkorobogatovZarhin2012}, relating the transcendental Brauer group of \(A\) with \'etale cohomology and the Chow group of the base change \(\overline{A}\).
This formula allows some computations of transcendental Brauer groups, see for example \cite{SkorobogatovZarhin2012}.

While we cannot say anything about \(\Br^r(A)\) itself, we prove an analogue of the above formula for higher Brauer groups and discuss how some of the computations of transcendental Brauer groups carry over to higher Brauer groups.

\subsection*{Notations and Conventions.}
For a field \(k\) we fix a separable closure \(\overline{k}\) and denote by \(\GalGrp_k=\Gal(\overline{k}/k)\) its absolute Galois group.
For each \(k\)-variety \(X\) we denote by \(\overline{X}=X\times_k\overline{k}\) the base change to the closure.

For an abelian group \(G\) we denote by \(G[n]\) the kernel of the multiplication by \(n\).
Let \(\ell\) be a prime; the subgroup \(G\{\ell\}=\varinjlim_nG[\ell^n]\) is referred to as the \(\ell\)-primary torsion subgroup.
Taking the projective limit instead gives us the Tate module \(\T_\ell G=\varprojlim_n G[\ell^n]\) which is a torsion free \(\Z_\ell\)-module.

Similarly, for an abelian variety \(A\) over \(k\) we denote by \(A[n]\) the kernel of the multiplication by \(n\) in \(\mathrm{End}(\overline{A})\).

For a group \(G\) and a \(G\)-module \(M\) we define \(M^G\) to be the maximal subgroup of \(M\) being invariant under the action of \(G\).

\section{Lichtenbaum cohomology}
We start with the definitions of Lichtenbaum cohomology, higher Brauer groups and Chow-L groups as used in this paper.

For the entire section let \(X\) be an equi-dimensional smooth quasi-projective variety over a field \(k\) and denote by \(z^n(X,\bullet)\) the cycle complex of abelian groups defined by Bloch \cite{Bloch1986}.
As the assertion
\[z^n(-,\bullet):U\mapsto z^n(U,\bullet)\]

defines a complex of sheaves on the flat site of \(X\), it also defines complexes \(z^n(-,\bullet)_\et\) resp. \(z^n(-,\bullet)_\mathrm{Zar}\) of sheaves on the \'etale resp. Zariski site of \(X\).
We remark, that these complexes are unbounded on the left.
For an abelian group \(A\), we set \(A_X(n)_\et=(z^n(-,\bullet)_\et\tensor A)[-2n]\) and define the Lichtenbaum respectively motivic cohomology groups with coefficients in \(A\) by letting
\begin{align*}
  &\hl{m}(X,A(n)):=\mathbb{H}_\et^m(X,A_X(n)_\et),\quad\textrm{resp.}\\
  &\hm{m}(X,A(n)):=\mathbb{H}_\textrm{Zar}^m(X,A_X(n)_\textrm{Zar})\,.
\end{align*}
The definition of the hypercohomology of an unbounded complex was given in \cite[p. 121]{SuslinVoevodsky2000}.

For Chow groups we have the isomorphisms \(\CH^r(A)\isom\hm{2r}(X,\Z(r))\) which motivates the definition \(\CHL^r(A):=\hl{2r}(X,\Z(r))\) of Chow-L groups.
With rational coefficients motivic cohomology and Lichtenbaum cohomology agree and therefore \(\CHL^r(A)\tensor\Q\isom\CH^r(A)\tensor\Q\).

We define higher Brauer groups as \(\Br^r(X):=\hl{2r+1}(X,\Z(r))\).
The term `higher Brauer groups' is justified by the fact that in weight \(n=1\) we have a quasi-isomorphism \(\Z_X(1)_\et\sim\mathbb{G}_m[-1]\) which implies that \(\Br^1(X)\) and \(\Br(X)\) are isomorphic.

We remark that with rational coefficients Lichtenbaum cohomology groups and motivic cohomology groups coincide and therefore \(\hl{2r+1}(X,\Q(r))\) is trivial.
This implies that the higher Brauer groups \(\Br^r(X)\) is a quotient of \(\hl{2r}(X,\Q/\Z(r))\) and thus a torsion group.

For each prime \(\ell\) multiplication by \(\ell^m\) gives us the exact sequence
\begin{align*}
  0\to \Z_X(n)_\et\overset{\ell^m}\to \Z_X(n)_\et\to(\Z/\ell^m\Z)_X(n)_\et\to0\,.
\end{align*}
For primes \(\ell\) different from the characteristic of \(k\) this sequence together with the quasi-isomorphism \((\Z/\ell^m\Z)_X(n)_\et\sim\tensorpow{\mu_{l^m}}{n}\) constructed by Geisser and Levine \cite[Theorem 1.5]{GeisserLevine2001} furnishes us with short exact universal coefficient sequences
\begin{align*}
  0\to\hl{i}(X,\Z(n))\tensor\Z/\ell^m\Z\to\h{i}(X,\tensorpow{\mu_{\ell^m}}{n})\to\hl{i+1}(X,\Z(n))[{\ell^m}]\to0\,.
\end{align*}
In particular, in bidegree \((i,n)=(2r,r)\) this sequence reads
\begin{align}\label{eq:coeffseqbr}
  0\to\CHL^r(X)\tensor\Z/\ell^m\Z\to\h{2r}(X,\tensorpow{\mu_{\ell^m}}{r})\to\Br^r(X)[{\ell^m}]\to0
\end{align}
and generalises the exact sequence for the classical Brauer group induced by the Kummer sequence.




\section{Tate Conjecture for Abelian Varieties}
\label{sec:tateobst}
Let \(k\) be a field finitely generated over \(\Q\).
Fix a separable closure \(\overline{k}\) of \(k\) and denote by \(\GalGrp_k:=\mathrm{Gal}(\overline{k}/k)\) the absolute Galois group.
For a smooth projective geometrically integral \(k\)-variety we consider the composition of canonical homomorphisms
\begin{align}
  \label{eq:compositioncycle}
  \CHL^r(X)\to\CHL^r(\overline{X})\to\hl{2r}(\overline{X},\Z/\ell^n\Z(r))\overset\isom\to\h{2r}(\overline{X},\mu_{\ell^n}^{\otimes r})
\end{align}
which induces a bunch of homomorphisms
\begin{align*}
c^r_{\ell^n}&:\CHL^r(X)\to\h{2r}(\overline{X},\mu_{\ell^n}^{\otimes r})^{\GalGrp_k}\\
\varprojlim_nc^r_{\ell^n}&:\CHL^r(X)\to\h{2r}(\overline{X},\Z_\ell(r))^{\GalGrp_k}\\
\overline{c}^r_{\ell^n}&:\CHL^r(X)\otimes\Z/\ell^n\Z\to\h{2r}(\overline{X},\mu_{\ell^n}^{\otimes r})^{\GalGrp_k}\\
\varprojlim_n\overline{c}^r_{\ell^n}&:\CHL^r(X)^\wedge\to\h{2r}(\overline{X},\Z_\ell(r))^{\GalGrp_k}
\end{align*}
where \(\CHL^r(X)^\wedge:=\varprojlim_n\CHL^r(X)\otimes\Z/\ell^n\Z\).
The map \(\varprojlim_nc^r_{\ell^n}\) induces a map
\[\varprojlim_nc^r_{\ell^n}\otimes\mathrm{id}_{\Z_\ell}:\CHL^r(X)\otimes\Z_\ell\to\h{2r}(\overline{X},\Z_\ell(r))^{\GalGrp_k}\]
and the image of \(\varprojlim_nc^r_{\ell^n}\otimes\mathrm{id}_{\Z_\ell}\) agrees with the image of \(\varprojlim_n\overline{c}_{\ell^n}^r\).
Moreover, the homomorphism \(\varprojlim_n c_{\ell^n}^r\otimes\mathrm{id}_{\Q_\ell}\) agrees the cycle map \(c_{\Q_\ell}^r:\CH^r(X)\otimes \Q_\ell\to\h{2r}(\overline{X},\Q_\ell(r))^{\GalGrp_k}\).

\begin{theorem}
  \label{thm:tateobst}
  Let \(A\) be an abelian variety over a finite field extension \(k\) of \(\Q\).
  Fix an integer \(r\) and a prime \(\ell>2r\).
  If the group \(\Br^r(A)\{\ell\}\) is finite, then the Tate conjecture \(\mathrm{TC}^r(A)_{\Q_\ell}\) for \(A\) in codimension \(r\) holds at the prime \(\ell\).
\end{theorem}
\begin{proof}
  We have the following commutative diagram with exact rows.
    \begin{align*}
    \xymatrix@C=0.5cm@R=0.5cm{
      \CHL^r(A)\otimes\Z/{\ell^n}Z\ar[r]\ar@{=}[d]&\h{2r}(A,\tensorpow{\mu_{\ell^n}}{r})\ar[r]\ar^{\rho^{2r,r}_{\ell^n}}[d]&\Br^r(A)[l^n]\ar[d]\ar[r]&0\\
      \CHL^r(A)\otimes\Z/{\ell^n}Z\ar^-{\overline{c}_{\ell^n}^r}[r]&\h{2r}(\overline{A},\mu_{\ell^n}^{\otimes r})^{\GalGrp_k}\ar[r]&\coker(\overline{c}_{\ell^n}^r)\ar[r]&0
    }
  \end{align*}

  Skorobogatov and Zarhin have shown \cite[Proposition\,2.2]{SkorobogatovZarhin2012} that with our assumptions the canonical homomorphism \(\rho_{\ell^n}^{2r,r}\) admits a section and hence is surjective.
  This implies that for each integer \(n\) the group \(\coker(\overline{c}_{\ell^n}^r)\) is quotient of \(\Br^r(A)[\ell^n]\).
  
  Since \(\Br^r(A)\{\ell\}=\bigcup_m\Br^r(A)[\ell^m]\) is finite, there exists an integer \(N\) such that \(\Br^r(A)\{\ell\}=\Br^r(A)[\ell^{N}]\).
  Since for all \(n\) the cardinality of the cokernels is bounded by \(\#\coker(\overline{c}_{\ell^n}^r)\leq\#\Br^r(A)[\ell^N]\), it follows that the limit \(\varprojlim_n\coker(\overline{c}_{\ell^n}^r)\) is a finite group.

  This implies that the cokernel of \(\varprojlim_n\overline{c}_{\ell^n}^r\) and hence the cokernel of \(\varprojlim_nc_{\ell^n}^r\otimes\mathrm{id}_{\Z_\ell}\) is finite.
  But this means that \(\varprojlim_nc_{\ell^n}^r\otimes\mathrm{id}_{\Q_\ell}\) and thus \(c_{\Q_\ell}^r\) are surjective.
\end{proof}

\begin{remark}
  With similar arguments Rosenschon and Srinivas \cite[Theorem\,1.4]{RosenschonSrinivas2013} have proven that for any smooth projective, geometrically connected variety \(X\) over a finite field \(\mathbb{F}_q\) and a prime \(\ell\neq\operatorname{char}\mathbb{F}_q\) the Tate conjecture \(\operatorname{TC}^r(X)_{\Q_\ell}\) for \(X\) at \(\ell\) in codimension \(r\) holds precisely when the group \(\Br^r(X)\{\ell\}\) is finite.
  Moreover, they constructed an integral cycle map
  \[c_{\Z_\ell}^r:\CHL^n(X)\otimes\Z_\ell\to\h{2r}(\overline{X},\Z_\ell(n))^{\GalGrp_{\mathbb{F}_q}}\]
  and showed that \(\TC^r(X)_{\Q_\ell}\) is equivalent to the surjectivity of \(c_{\Z_\ell}^r\).
  Note that the latter is not true when the Chow-L group is replaced by the Chow group.
  More precisely, there exists examples where the cycle map \(\CH^r(X)\otimes\Z_\ell\to\h{2r}(\overline{X},\Z_\ell(n))^{\GalGrp_{\mathbb{F}_q}}\) is not surjective; see \cite[Section\,2]{CTSzamuely2010}.

  So far the author is not aware of an answer to the question whether the converse of the above theorem holds.
  We also do not know whether in the situation of the Theorem the Tate conjecture implies surjectivity of an integral cycle map such as \(c_{\Z_\ell}^r\).

  Let \(k\subseteq\mathbb{C}\) is a field finitely generated over \(\Q\) and let \(X\) be a smooth projective geometrically integral \(k\)-variety.
  In that case Rosenschon and Srinivas \cite[Theorem\,1.2]{RosenschonSrinivas2013} have shown that the Tate conjecture for \(\overline{X}\) in codimension \(r\) at prime \(\ell\) is equivalent to the surjectivity of the integral cycle map
  \[\CHL^r(\overline{X})\otimes\Z_\ell\to\h{2n}(\overline{X},\Z_\ell(r))^{\GalGrp_k}\,.\]
\end{remark}

\section{Computing Higher Brauer Groups}
In view of the previous section it is desirable being able to compute higher Brauer groups, at least for abelian varieties.
The objective of this section is to extend efforts of computing transcendental Brauer groups of abelian varieties in characteristic zero to higher Brauer groups.

The main result of this section is the following theorem which for \(r=1\) was shown by Skorobogatov and Zarhin \cite[Corollary 2.3]{SkorobogatovZarhin2012}.

\begin{theorem}\label{thm:genskoza2.2}
  Let \(k\) be a field finitely generated over \(\Q\) and let \(A\) be an abelian variety over \(k\).
  For any integers \(i,n\geq1\) such that \((n,2r!)=1\) we have an isomorphism
  \[\Br^r(A)[n]\big/\Br_\mathrm{alg}^r(A)[n]\isom\h{2r}\left(\overline{A},\tensorpow{\mu_n}{r}\right)^{\GalGrp_k}\big/\left(\CHL^{r}\left(\overline{A}\right)\tensor\Z/n\Z\right)^{\GalGrp_k}\,.\]
\end{theorem}
\begin{proof}
  Restricting the exact sequence (\ref{eq:coeffseqbr}) with \(X=\overline{A}\) to the Galois invariants and using the exact sequence (\ref{eq:coeffseqbr}) with \(X=A\) we get the following commutative diagram with exact rows
  \begin{align}\label{diag:skozar}
    \xymatrix@C=0.3cm@R=0.5cm{
      0\ar[r]&\left(\CHL^{i}\left(\overline{A}\right)\tensor\Z/n\Z\right)^{\GalGrp_k}\ar[r]&\h{2i}\left(\overline{A},\tensorpow{\mu_n}{i}\right)^{\GalGrp_k}\ar^-\psi[r] & \Br^{i}(\overline{A})[n]^{\GalGrp_k}\\
      &&\h{2i}\left(A,\tensorpow{\mu_n}{i}\right)\ar[r]\ar^\alpha[u] & \Br^{i}(A)[n]\ar^\beta[u]\ar[r]&0\,.
    }
  \end{align}
  By \cite[Proposition\,2.2]{SkorobogatovZarhin2012} the homomorphism \(\alpha\) is split surjective.
  In that case the commutativity of the diagram implies that the images of \(\psi\) and \(\beta\) coincide which implies that the groups \(\Br^i(A)[n]/\Br_\mathrm{alg}^i(A)[n]\) and \(\h{2i}(\overline{A},\tensorpow{\mu_n}{i})^{\GalGrp_k}/(\CHL^{i}(\overline{A})\tensor\Z/\Z n)^{\GalGrp_k}\) are isomorphic.
\end{proof}

For products of elliptic curves we can prove even more precise formulae.
Let \(E_0,\dots, E_r\) be elliptic curves and set \(A=E_0\times E_1\times\dots\times E_r\).
Our aim is to obtain a better understanding of the higher Brauer group \(\Br^r(A)\).
For this we need some results concerning the Chow-L group \(\CHL^r(\overline{A})\) and the \'etale cohomology groups \(\h{2r}(\overline{A},\tensorpow{\mu_n}{r})\).

The map \(c^r_n:\CH^r(\overline{A})\otimes\Z/n\Z\to \h{2r}(\overline{A},\mu_n^{\otimes r})\) induced by the cycle map factors through the map \(\overline{c}_{n}^r:\CHL^r(\overline{A})\otimes\Z/n\Z\to\h{2r}(\overline{A},\mu_n^{\otimes r})\) which is induced by the composition of caninical homomorphisms (\ref{eq:compositioncycle}).

The isomorphisms \(\Hom(E_i[n],\Z/n\Z)\overset\isom\to\h1(\overline{E_i},\Z/n\Z)\) and the perfect Weil pairings furnish us with canonical isomorphisms (see \cite[Section 3]{SkorobogatovZarhin2012})
\begin{align*}
  \h1\left(\overline{E_i},\mu_n\right)\tensor\h1\left(\overline{E_j},\Z/n\Z\right)\overset\isom\to\Hom(E_i[n],E_j[n])\,. 
\end{align*}
Using these isomorphisms and the Künneth formula \cite[VI.~§8]{Milne1980} we obtain an isomorphism of \(\GalGrp_k\)-modules
  \begin{align*}
    (\Z/n\Z)^{r+1}\oplus\left(\bigoplus_{0\leq i<j\leq r}\Hom(E_i[n],E_j[n])\right) \overset\isom\to \h{2r}(\overline{A},\tensorpow{\mu_n}r)\,.
  \end{align*}

  Deeper investigation of the above isomorphism shows that each of the three \(\Z/n\Z\)-copies is generated by a fundamental class coming from one of the elliptic curves \(\overline{E_i}\).

\begin{lemma}\label{lm:z3inchl2}
  The group \(\CHL^r(\overline{A})\tensor\Z/n\Z\) contains a direct summand \((\Z/n\Z)^{r+1}\) generated by the classes coming from the elliptic curves.
\end{lemma}
\begin{proof}
  The group \(\CH^r(\overline{A})\otimes\Z/n\Z\) has the direct summand \((\Z/n\Z)^r\) which is generated by the classes \([\overline{E_i}]\) of the elliptic curves.
  The composition
  \[(\Z/n\Z)^{r+1}\to\CH^r(\overline{A})\tensor\Z/n\Z\to\h{2r}(\overline{A},\tensorpow{\mu_n}{r})\to(\Z/n\Z)^{r+1}\,,\]
  where the first map is \((\lambda_0,\lambda_1,\dots,\lambda_r)\mapsto \sum_{i=0}^r\lambda_i[E_i]\) and the last map is the obvious projection to a direct summand, is the identity map.
  Since the involved cycle map and thus the whole composition factors through \(\CHL^r(\overline{A})\tensor\Z/n\Z\), this implies that \(\CHL^r(\overline{A})\tensor\Z/n\Z\) has a direct summand \((\Z/n\Z)^{r+1}\) generated by classes coming from the elliptic curves. 
\end{proof}

In the case \(r=1\) the following theorem and its applications presented below were shown by Skorobogatov and Zarhin \cite{SkorobogatovZarhin2012}.

\begin{theorem}\label{thm:br2ellcurves}
  Let \(R_{n}\) be the quotient \(\left(\CHL^r(\overline{A})\tensor\Z/n\Z\right)\big/\left(\Z/n\Z\right)^{r+1}\) and let \(H_{n}\) be the \(\GalGrp_k\)-module \(\bigoplus_{0\leq i< j\leq r}\Hom(E_i[n],E_j[n])\).
  Then:
  \begin{enumerate}[(i)]
  \item For each integer \(n\) there exists an isomorphism
    \[H_{n}/R_{n}\overset\isom\to \Br^r(\overline{A})[n]\,.\]
  \item For each integer \(n\) not divisible by  primes \(\ell\leq r\) we get an isomorphism
    \[\Br^r(A)[n]/\Br^r_\mathrm{alg}(A)[n]\isom H_{n}^{\GalGrp_k}/R_{n}^{\GalGrp_k}\,.\]
  \end{enumerate}
\end{theorem}
\begin{proof}
The exact sequence (\ref{eq:coeffseqbr}) with \(X=A\) and \(i=2\) stays exact when we remove the \((\Z/n\Z)^{r+1}\) summands in the first two terms, i.e. we have an exact sequence \(0\to R_{n}\to H_{n}\to\Br^r(\overline{A})[n]\to0\) proving (i).

For (ii) consider the commutative diagram with exact rows
\begin{align}
  \xymatrix@R=0.5cm@C=0.5cm{
    0\ar[r]&R_{n}^{\GalGrp_k}\ar[r]& H_{n}^{\GalGrp_k}\ar^\psi[r] & \Br^{r}(\overline{A})[n]^{\GalGrp_k}\\
    &&\h{2r}(A,\tensorpow{\mu_n}{r})\ar[r]\ar^\alpha[u] & \Br^{r}(A)[n]\ar^\beta[u]\ar[r]&0
  }
\end{align}
where the vertical arrow \(\alpha\) is the composition of the canonical maps \(\h4(A,\tensorpow{\mu_n}r)\to\h{2r}(\overline{A},\tensorpow{\mu_n}r)^{\GalGrp_k}\) and \(\h{2r}(\overline{A},\tensorpow{\mu_n}r)^{\GalGrp_k}\to H_{n}^{\GalGrp_k}\).
If \(n\) is not divisible by a prime less or equal \(r+1\), the first arrow is split surjective by \cite[Proposition\,2.2]{SkorobogatovZarhin2012} and composing the section of the second with the section of the first arrow we get a section for \(\alpha\).
With the arguments of the proof of Theorem~\ref{thm:genskoza2.2} we get that \(\Br^r(A)[n]/\Br_\mathrm{alg}^r(A)[n]\) is isomorphic to the group \(H_{n}^{\GalGrp_k}/R_{n}^{\GalGrp_k}\).
\end{proof}

We now consider the special case of a triple product \(A=E\times E\times E\) of a elliptic curve \(E\) and \(r=2\).
The groups \(R_{n}\) and \(H_{n}\) are defined as in Theorem\,\ref{thm:br2ellcurves}.

\begin{proposition}
  Let \(E\) be an elliptic curve such that for each prime \(\ell\neq 2,3\) the image of \(\GalGrp_k\) in \(\mathrm{Aut}(E_\ell)\) is \(\mathrm{GL}(2,\F_\ell)\).
  Then  \(\Br^2(A)\{\ell\}=\Br^2_\mathrm{alg}(A)\{\ell\}\).
\end{proposition}
\begin{proof}
  First note, that as \(\mathrm{GL}(2,\F_\ell)\) is not solvable for \(\ell\geq5\), the elliptic curve of the proposition is without complex multiplication, i.e. \(\mathrm{End}(\overline{E})=\Z\), c.f. an argument given in \cite[Proposition\,4.1]{SkorobogatovZarhin2012}.

  In view of Theorem~\ref{thm:br2ellcurves} it suffices to prove that \(H_{n}^{\GalGrp_k}/R_{n}^{\GalGrp_k}\) is trivial.
Application of \cite[Lemma 4.4]{SkorobogatovZarhin2012} to \(H_{n}^{\GalGrp_k}\) yields that \(H_{n}^{\GalGrp_k}\isom (\Z/\ell^n\Z)^3\) (generated by the identity endomorphisms in the copies of \(\mathrm{End}(\overline{E})\) in \(H_{n}\)).
It therefore suffices to prove that \(R_{n}^{\GalGrp_k}\isom(\Z/\ell^n\Z)^3\).

  Consider the composition of canonical maps
  \begin{align*}
    (\mathrm{End}(\overline{E})/n)^3\overset{\phi}\to&\CH^2(\overline{A})/n\overset{\iota}\to\CHL^2(\overline{A})/n\\\overset{c_n^2}\to&\h4(\overline{A},\tensorpow{\mu_n}2)\to\mathrm{End}(E[n])^3
  \end{align*}
  where \(\phi\) is given by \[(f,g,h)\mapsto [\{(x,f(x),0)\}]+[\{(x,0,g(x))\}]+[\{(0,x,h(x))\}]\,.\]
The two outer groups are \((\Z/n\Z)^3\) (generated by the identity maps) and one easily checks that this composition is the identity.
This implies that \(\iota\phi\) is injective and, as its image is contained in \(R_{n}\), that \(R_{n}\) contains \((\Z/n\Z)^3\).
\end{proof}

For our second application  let \(E\) be an elliptic curve over \(\Q\) with complex multiplication.
For an odd prime \(\ell\) the curve \(E\) has no rational isogeny of degree \(\ell\), if and only if \(E[\ell]\) does not contain a Galois invariant subgroup of order \(\ell\).

\begin{proposition}
  Let \(\ell\) be an odd prime such that \(E\) has no rational isogeny of degree \(\ell\).
  Then the group \(\Br^2(\overline{A})[\ell]^{\GalGrp_\Q}\) is trivial.
\end{proposition}
\begin{proof}
  The \(\GalGrp_\Q\)-module \(\h4(\overline{A},\tensorpow{\mu_n}2)\isom(\Z/\ell\Z)^3\oplus\mathrm{End}(E_\ell)^3\) is semisimple since \(\mathrm{End}(E[\ell])\) is semisimple.
  This implies that the exact sequence (\ref{eq:coeffseqbr}) yields an isomorphism \(\h4(\overline{A},\tensorpow{\mu_n}2)\isom (\CHL^2(\overline{A})\tensor\Z/n\Z)\oplus \Br^2(\overline{A})[\ell]\).

  Denote by \(G_\ell\) the image of \(\GalGrp_\Q\) in \(\mathrm{End}(E[\ell])\).
  It is known that the centraliser of \(G_\ell\) in \(\mathrm{End}(E[\ell])\) is isomorphic to \(\Z/\ell\Z\) \cite[Theorem 4.6]{SkorobogatovZarhin2012}.
  On the other hand, the centraliser equals the subgroup \(\mathrm{End}(E[\ell])^{\GalGrp_\Q}\) of Galois invariant endomorphisms and contains all the \(\Z/\ell\Z\) multiples of the identity endomorphism which come from the multiples of the three diagonals in \(\overline{A}\).
  Thus, we have the inclusion \(\h4(\overline{A},\tensorpow{\mu_\ell}2)^{\GalGrp_\Q}\subseteq \CHL^2(\overline{A})\tensor\Z/\ell\Z\) proving the triviality of \(\Br^2(\overline{A})[\ell]\).
\end{proof}

\section*{Acknowledgement}
The author thanks Andreas Rosenschon for sugesting to work on these questions.
Parts of this paper were written while the author was member of DFG Sonderforschungsbereich~701 at Bielefeld University.

\bibliography{abel3.bib}
\bibliographystyle{amsplain}

\end{document}